\documentclass{amsart}
\usepackage[mathscr]{eucal}
\usepackage{color}
\usepackage{xypic}
\usepackage{amsfonts}   
\usepackage{amsmath}
\usepackage{amsthm}
\usepackage{amssymb}
\usepackage{latexsym}
\usepackage[english]{babel}
\usepackage[utf8]{inputenc}

\let\nc\newcommand

\nc{\la}{\label}

\newtheorem{conjecture}{Conjecture}

\newtheorem{theorem}{Theorem}[section]
\newtheorem{definition}[theorem]{Definition}

\newtheorem{lemma}[theorem]{Lemma}
\newtheorem{proposition}[theorem]{Proposition}
\newtheorem{example}[theorem]{Example}
\newtheorem{remark}[theorem]{Remark}
\newtheorem{problem}{Problem}
\def\k{\mathsf k}
\def\C{\mathbb C}

\newtheorem{corol}{Corollary}

\AtEndDocument{\bigskip{\footnotesize%
 
  \addvspace{\medskipamount}
   J.\,Schwarz \textsc{Instituto de Matematica e Estatistica, Universidade de S\~{a}o Paulo, \\ S\~{a}o Paulo, Brasil} \par
  \textit{E-mail address}: \texttt{joao@ime.usp.br }  \par
   \addvspace{\medskipamount}
  
}}

\begin{document}

\title[A Poisson analogue of Noether's Problem]{A Poisson analogue of Noether's Problem}

\author{Jo\~ao Schwarz}

\begin{abstract}
In this paper we show that the Poisson analogue of the Noether's Problem has a positive solution for essentially all finite symplectic reflection groups --- the analogue of complex reflection groups in the symplectic world. Our proofs are constructive, and generalize and refines previously known results. As an interesting consequence of the solution of this problem for complex reflection groups, we obtain the Poisson rationality of the Calogero-Moser spaces associated to any complex reflection group. The results of this paper can be thought as analogues of the Noncommutative Noether Problem (cf. \cite{AD1}, \cite{FS}) and the Gelfand-Kirillov Conjecture for rational Cherednik algebras (cf. \cite{Etingof}) in the 'quasi-classical limit'. In the second half of the paper, an abstract framework to understand these results is introduced, and it is shown that every Coloumb branch of a $3d \, \mathcal{N}=4$ gauge theory is Poisson rational as an application. We also obtain the Gelfand-Kirillov Conjecture for trigonometric Cherednik algebras and the Poisson rationality of their trigonometric Calogero-Moser spaces at the same time.




\end{abstract}

\maketitle

\section{Introduction}

We assume a base field $\k$ with $char \, \k=0$. All rings in discussion are algebras over this field.

If $A$ is a Poisson domain and $S$ any multiplicatively closed subset of $A$, then the localization $A_S$ has natural structure of Poisson domain, uniquely extending the one from $A$ : \[ \{as^{-1}, bt^{-1} \} = \{a, b\}s^{-1}t^{-1} - \{a, t\}bs^{-1} t^{-2} - \{s, b\}as^{-2}t^{-1} + \{s, t\}abs^{-2}t^{-2},\]  $a,b \in A, s,t \in S$.

In virtue of this simple fact, a question that has attracted a considerable amount of interest is the following: given two distinct Poisson domains $A, A'$, when are their field of fractions isomorphic as Poisson algebras, i.e., when they are \emph{Poisson birational}? In case $X$ and $Y$ are two affine Poisson varieties, we say that they are \emph{Poisson birationally equivalent} in case $\mathcal{O}(X)$ is Poisson birational to $\mathcal{O}(Y)$. One of the most significant Poisson domains in these questions is the following:

\begin{definition}
We consider the following Poisson algebra: $\k[x_1,\ldots,x_n, y_1, \ldots, y_n ]$, with Poisson brackets defined by $\{x_i, x_j \} = \{y_i, y_j \} =0, \{x_i, y_j\} = \delta_{ij},  1 \leq i,j,\leq n$. Following \cite{Dumas}, \cite{Goodearl}, we call it the \emph{Poisson-Weyl algebra}, and denote it by $P_n$
\end{definition}

$P_n$ is, of course, the standard Poisson structure on the symmetric algebra of a non-degenerate symplectic vector space. We shall denote $K_n$ the Poisson field obtained from $P_n$ by localizing the set of all non-zero elements.

A version of the important Gelfand-Kirillov Conjecture (\cite{Gelfand}), relating the field of fractions of envelopping algebras of finite dimensional algebraic Lie algebras with the field of fractions of the Weyl algebras, has been considered for the Poisson algebra $S(\mathfrak{g})$ of a finite dimensional Lie algebra $\mathfrak{g}$ with the Konstant-Kirillov bracket; namely, when its field of fractions is Poisson birational with the field of fractions of a Poisson-Weyl Algebra. It was shown to be true for nilpotent $\mathfrak{g}$ by Vergne (\cite{Vergne}), and later generalized to algebraic solvable Lie algebras by Tauvel and Yu (\cite{Tauvel}), for algebraically closed fields. The question is also discussed in \cite{NW}. In a similar way, versions of the quantized Gelfand-Kirillov Conjecture (\cite{Brown} I.2.11, II.10.4)  for certain Poisson algebras have been considered by Goodearl and Launois (\cite{Goodearl}), and Launois and Lecoutre (\cite{Launois}).

In this paper, we are interested in a similar question, considered by Julie Baudry \cite{Baudry} and François Dumas \cite{Dumas}, which is a Poisson analogue of the Noether's Problem (\cite{Noether}) about the rationality of invariants of the field of rational functions:


\begin{problem}
\emph{(Poisson Noether's Problem)}
Let $G$ be a finite group of linear Poisson automorphisms of $P_n$, i.e., those Poisson automorphisms that fix the subspace $\k x_1 \oplus \ldots \oplus\ k x_n \oplus \k y_1 \oplus \ldots \oplus \k y_n$ --- and hence can be naturally seen as a subgroup of $SP_{2n}(\k)$. When we have $K_n^G \simeq K_n$ as Poisson fields?
\end{problem}

\begin{definition}
Whenever we have a Poisson isomorphism of a Poisson field with an certain $K_n$, we will call it Poisson rational. In case we have $X$ is an affine Poisson variety and $Frac \, \mathcal{O}(X)$ is Poisson rational, we also call $X$ Poisson rational.
\end{definition}

A particularly important case of linear Poisson automorphisms arise in the following way (cf. \cite{Dumas}):

\begin{definition}
Let $G \subset GL_n(\k)$ be a finite group of automorphisms acting an $n$ dimensional vector space $V$, and consider its diagonal action on $W=V^* \oplus V$, where $\k[V^*]=:\k[x_1,\ldots,x_n]$ and $\k[V] =: \k[y_1,\ldots,y_n]$. Then $G$ acts by Poisson automorphisms on $\k[W] = \k[x_1,\ldots,x_n, y_1, \ldots, y_n]$. In this case we call the action \emph{diagonally linear}.
\end{definition}

In \cite{Baudry}, it was  proved that Poisson Noether's Problem has positive solution in case $\k = \mathbb{C}$ and $G$ is a finite group of $SL_2(\mathbb{C})$ acting on $P_1$ and for diagonally linear action of the Weyl group $B_2$ on $P_2$.
So, in particular, the Kleinian singularities are Poisson rational. In \cite{Dumas}, for arbitrary $\k$, a positive solution was found when $G$ acts diagonally linear on $P_n$ and arises from a representation that decomposes in a direct sum of one dimensional components --- in particular, when $G$ is abelian and $\k$ is algebraically closed. An important aspect of these results is an explict exhibition of the isomorphism $K_n^G \simeq K_n$ as Poisson fields. We remark that in \cite{Baudry}, a more general version of the Poisson Noether's Problem was also considered, with results on 'quasi-classical limits' of the quantum torus (\cite{Baudry}, \cite{Baudry2}).

We remark that before the aforementioned works, Iain Gordon, in an unpublisehd manuscript \cite{Gordon}, has already shown that, in fact, for any finite $\Gamma < SL_2(\mathbb{C})$, the Poisson Noether's Problem has a positive solution for the natural action of the wreath product $\Gamma \wr S_n$ on $\mathbb{C}^{2n}$, using the work of \cite{Wang} on symplectic resolutions of the quotient variety of this type. However, the argument does not allow immediatly an explicit description of the Poisson rationality, and the idea of using symplectic resolutions of singularities cannot be applied in general (\cite{Bellamy}).

In the first half of this paper, we generalize the work of Baudry and Gordon in finding an \emph{explict} description of the positive solution for the Poisson Noether's Problem for (essentially) all non-exceptional indecomposable sympletic reflection groups (recalled in Section 2.1).

Our first main result is the following

\begin{theorem}\label{main-1}
Let $\k=\C$ and $G$ an irreducible complex reflection group or the symmetric group $S_n$ with a diagonally linear action on $P_n$. Then:
\begin{enumerate}
\item Poisson Noether's Problem has a positive solution.
\item We can exhibit explicitely the isomorphism of Poisson fields $K_n^G \simeq K_n$.
\end{enumerate}
\end{theorem}

With this result we show:

\begin{corol} \label{corol-1}
Let $G$ be any complex reflection group. Then Poisson Noether's Problem has a positive solution.
\end{corol}

Using the result for the symmetric group and work of Baudry for finite subgroups of $SL_2(\mathbb{C})$, we can refine Gordon's result and show:

\begin{corol} \label{corol-2}
For $\k = \C$ and $G = \Gamma \wr S_n$ a wreath product, we have  an explicit Poisson isomorphisms $K_n^G \simeq K_n$
\end{corol}

Theorem \ref{main-1} and Corollary \ref{corol-2} cover essentially all indecomposable symplectic reflection groups (cf. Section 2.1).

As a second main result, we show:

\begin{theorem} \label{main-2}
The Calogero-Moser space associated to any finite complex reflection group is Poisson rational.
\end{theorem}

The second half of this paper builts on the fact that the cases of isomorphism between Poisson fields discussed above are the 'quasi-classical limit' of similar questions on the noncommutative Ore domains that quantize the relevant Poisson algebras. For instance, Poisson Noether's Problem is the 'quasi-classical limit' of Noncommutative Noether's Problem (\cite{AD1}). In \cite{FS}, Thm. 1.2, the author together with V. Futorny, have shown that, given an affine complex variety $X$ with the action of a finite group $G$ such that $X/G$ is birationally equivalent to an affine variety $Y$, then $Frac \, D(X)^G \simeq Frac \, D(Y)$. In Theorem \ref{quasi-classical}, we prove a version of this result in the quasi-classical limit, which provides an abstract understanding of our approach in proving Theorem \ref{main-1}. Following these ideas, as applications we prove our third and fourth main results:

\begin{theorem} \label{main-3}
Every Coulomb branch of a $3d \, \mathcal{N}=g$ gauge theory, with symplectic representation of cotangent type, is Poisson rational.
\end{theorem}

\begin{theorem} \label{main-4}
Gelfand-Kirillov Conjecture holds for all trigonometric Cherednik algebras, and the trigonometric Calogero-Moser spaces are Poisson rational.
\end{theorem}

The structure of the paper is as follows. In Section 2 we discuss the relevant preliminaries on sympletic reflection groups, Cherednik algebras, Calogero-Moser spaces and Coulomb branches. In Section  3.1 we prove Theorem \ref{main-1} for irreducible complex reflection groups, following the idea in \cite{FS} in the 'quasi-classical limit'. The part (2) of this Theorem is Proposition \ref{part2}. In Section 3.2, as applications, we prove Corollaries \ref{corol-1} and \ref{corol-2}, and Theorem \ref{main-2}. In Section 4.1 we discuss the connection between birational equivalence and Poisson birational equivalence of cotangent bundles, which provides an abstract understanding of the proof of Theorem \ref{main-1} --- cf. Theorem \ref{quasi-classical}. In Section 4.2, as applications, we prove Theorem \ref{main-3}  and Theorem \ref{main-4} (in a slightly more general form).

\section*{Acknowledgments}

J. S. is supported by FAPESP grant 2018/8146-5. The author is grateful to Vyacheslav Futorny and Farkhod Eshmatov for very fruitfull discussions regarding this paper, and to the Sichuan University for its support and hospitality during his visit. He is also grateful to M. Soares and M. C. Cardoso.

\section{Preliminaries}

\emph{Through the rest of the paper}, the base field is the complex numbers.

\subsection{Symplectic reflection groups and generalized Calogero-Moser spaces}
Lets recall the important notion of sympletic reflection groups, which are analogues for vector spaces with a non-degenerate sympletic form of a complex reflection group.

\begin{definition}
Let $V$ be a complex vector space of dimension $2n$, with a non-degenerate skew-symmetric form $\omega$. Let $\Gamma$ be a finite subgroup of $SP_{2n}(\C)$ generated by \emph{sympletic reflections}: that is, elements $g \in \Gamma$ such that $1-g$ has rank two. Then $\Gamma$ is called a \emph{sympletic reflection group}.
\end{definition}

We call the data above a \emph{sympletic triple} and denote it by $(V, \omega, \Gamma)$.

Their importance lies, among other things, in the study of sympletic singularities and sympletic reflection algebras; the later were introduced by Etingof and Ginzburng in \cite{Etingof}, and the former by Beauville in \cite{Beauville}. It was shown in \cite{Verbitsky} that a quotient of a  complex sympletic vector space $V$ by a finite subgroup $\Gamma$ of $SP_V(\C)$ has a symplectic resolution if and only if $\Gamma$ is a sympletic reflection group.

We now discuss our two main sources of sympletic reflection groups.

\begin{enumerate}
\item

Let $W$ be a finite complex reflection group acting on a vector space $h$ and in its dual $h^*$ in contragradient manner. Then if calling $V = h \oplus h^*$ we define a non-degenerate symplectic form:

\[ \omega((y,f),(u,g)) = g(y) - f(u), y,u \in h, \, f,g \in h^*, \]

then we have that the diagonal action of $W$ turns it into a sympletic reflection group.

\item

The natural action of  the wreath product $\Gamma \wr S_n$ on $\C^{2n}$, $\Gamma < SL_2(\C)$ a finite group. 

\end{enumerate}

Like the case of complex reflection groups, we have a decomposition of a sympletic triple into indecomposable ones, i.e., triples $(V,\omega,\Gamma)$ such that $V$ cannot be expressed as a direct sum of two non-trivial $\Gamma$-stable subspaces $V_1, \ldots V_2$ with $\omega(V_1,V_2)=0$. They have been classified (\cite{GS}), and the bulk of the classication theorem says that, apart from certain exceptional cases, all indecomposable triples are of the form (1) or (2) above.

Now we remember the notion of Calogero-Moser spaces.

They were first considered by Kazhdan, Konstant and Sternberg (\cite{KKS}) and studied by Wilson (\cite{Wilson}) and many others (cf. \cite{Etingof2}), and generalized for any complex reflection group (the original case being $S_n$) in \cite{Etingof}. To introduce it we recall some basic notions of rational Cherednik algebras (\cite{Etingof}). They are a double step degeneration of the double affine Hecke algebras considered by Cherednik (\cite{Cherednik}). The first step is the so called trigonometric Chrednik algebra, discussed bellow.

Let $W$ be a finite complex reflection group acting on a complex vector space $h$. Let $S$ be the set of of reflections and, for each $s \in S$ take $\alpha_s \in h^*$ and $\alpha_s^\vee \in h$ such that $ \alpha_s$ is an eigenvector of $\lambda_s$ (the non-trivial eigenvalue of $s$ in $h^*$); and $\alpha_s  ^\vee$ is an eigenvector of $\lambda_s^{-1}$ (the non-trivial eigenvalue of $s$ in $h$). Normalize them such that in the natural pairing $h^* \times h \rightarrow \C$ we have $(\alpha_s, \alpha_s^\vee)=2$. Finally, let $c: S \rightarrow \mathbb{C}$ a conjugation invariant function. 

Consider the algebra $H_{c,t}(W,h), t \in \C$, with the following generators and relations: the quotient of $\mathbb{C}W \ltimes T(h \oplus h^*)$ by the relations:

\[ [x,x']=[y,y']=0; \, [y,x]=tx(y)-\sum_{s \in S} c_s(y, \alpha_s) (x, \alpha_s^\vee) s, \]
with $x,x' \in h^*, y,y' \in h$.

\begin{definition}
$H_{c,t}(W,h)$ is called the Rational Cherednik algebra.
\end{definition}

In case $t \neq 0$, the center of this algebra is just $\C$; but in case $t=0$, the center --- which we will denote by $Z_c(W,h)$ --- is quite big. In fact, the rational Cherednik algebra is a finite module over it (\cite{Etingof}).

We have

\begin{proposition}
\emph{(\cite{Etingof})}
$Z_c(W,h)$ is a finitely generated algebra without zero divisors. It is also a Poisson algebra.
\end{proposition}

\begin{definition}
The Calogero-Moser space, denoted $\mathcal{M}_c(W,h)$, is $Spec \, Z_c(W,h)$
\end{definition}

\subsection{Trigonometric Cherednik algebras}

As we discussed above, the trigonometric Cherednik algebra is also a degeneration of the double affine Hecke algebra. We will introduce them following the approach in \cite{Etingof3}. It can be shown that it coincides with the usual definition in \cite{Etingof}  (cf. \cite{EM}, Example 7.10).

Let $X$ be an affine complex variety and $G$ a finite group of automorphisms of it. If $X^g$ is the set of fixed points of $X$ under the action of $g \in G$, the components $Y$ of $X^g$ of codimension $1$ in $X$ are called reflection hypersurfaces.

We have a canonical surjection of $\mathcal{O}_X$-modules, for each such $Y$, $(\ddagger) \, \xi_Y: T \, X \rightarrow \mathcal{O}_X(Y)/ \mathcal{O}_X$ (cf. \cite{Etingof3}, Section 2.4), where $O_X(Y)$ is the sheaf of regular functions on $X \setminus Y$ with pole of order at most one in $Y$. Fix $t \in \C$ and $c$ a $G$-invariant complex valued function on the set $S= \{(g,Y) \}$,  $g \in G$, $Y \subset X^g$ a reflection hypersurface. We have the following generalized version of the Dunkl-Opdam operators for such pairs $X$ and $G$:

\[ D = t L_v + \sum_{(g,Y) \in S} \frac{2 c(g,Y)}{1-\lambda_{Y,g}}f_Y(x)(1-g); \, D \in D(X)_r[c] \rtimes G, \]
where $D(X)_r$ is the ring of differential operators on $X$ with rational coeficients; $\lambda_{g,Y}$ is the eigenvalue of $g$ in the conormal bundle of $Y$; $L_v$ is the Lie derivative, $v$ a vector field on $X$; $f_Y \in \Gamma(X, \mathcal{O}_X(Y))$ an element in the coset of $\xi_Y(v)$ in the map considered above ($\ddagger$). When $t=0$, $t L_v$ does not vanish, but turns into the classical momentum (cf. \cite{Etingof3}, Definition 2.12).

\begin{definition}\label{genCherednik}
$H_{t,c}(X,G)$ is the subalgebra of $D_r(X)$ generated by $\mathcal{O}(X)$, $G$ and the Dunkl-Opdam operators. It is called the Cherednik algebra of $X$ and $G$. When $W$ is a Weyl group and $H$ the corresponding torus, we have the \emph{trigonometric} Cherednik algebra, denoted by $\mathbb{H}_{t,c}(W)$.
\end{definition}

We remark that in \cite{Etingof3} a more general version of Cherednik algebras is considered, allowing twists on the ring of of differential operators and non-affine $X$. We will not need to work in this generality.

\begin{definition}
Let $e = \sum_{g \in G} g$ be the indempotent and consider $U_{t,c}(X,G):= e H_{t,c}(X,G)e$ and, in case of trigonometric Cherednik algebras, $e=\sum_{w \in W} w$, $\mathbb{U}_{t,c}(W):= e \mathbb{H}_{t,c}(W)e$. These are called the spherical subalgebras.
\end{definition}

When $t=0$, the center of $H_{0,c}(X,G)$ is big, and similarly to the classical situation of rational Cherednik algebras, it can be shown that:

\begin{proposition}
\emph{\cite{Etingof3}}
$Z(H_{0,c}(X,G))$ is a finitely generated algebra without zero divisors. It is also a Poisson algebra.
\end{proposition}

\begin{definition}
We denote $\mathcal{M}_c(X,G)$ the associated Poisson variety, and call it a general Calogero-Moser space. In the case of trigonometric Cherednik algebras, we write $\mathcal{M}_c(W)$ and call it the trigonometric Caloger-Moser space.
\end{definition}

\subsection{Coulomb branches}

Coulomb branches are a 'recipe' to obtain certain affine normal Poisson varieties given a $3d \, \mathcal{N}=4$ gauge theory, with data a complex  connected reductive group $G$ and a symplectic finite dimensional representation of cotangent type $M=N \oplus N^*$, $N$ a finite dimensional complex representation. These objects have long been studied in physics --- more precisely, in quantum field theory (cf. \cite{BF}), and recently, received a precise mathematical definition in case of cotangent type, begininging in the work of Nakajima \cite{N}, and a satisfactory definition completed by Braverman, Finkelberg and Nakajima \cite{BFN}. They have recently received a lot of attentation due to important connections to mathematical physics and applications in algebra (cf. \cite{BF}). We just outline here its definition and state its main properties, as the details are fairly technical --- we refer to \cite{BFN}.

Consider $\mathcal{R}_{G,M}$ the moduli space of triples $(\mathcal{P}, \phi, s)$, where $\mathcal{P}$ is a principal $G$-bundle on $D = Spec \, \C[[z]]$, the formal disk; $\phi: \mathcal{P} \rightarrow D^\times \times G$ is a local trivialization over the formal punctured disk $D^\times = Spec \, \C((z))$; $s$ a section of $\mathcal{P} \times_G N$ such that $\phi(s)$ is regular.  $\mathcal{R}_{G,M}$ is an ind-scheme, related to the affine Grassmanian of $G$ (cf. \cite{BFN}). In its equivariant Borel-Moore homology $H_*^{G_\mathcal{O}}$, where $G_\mathcal{O} = G[[z]]$, one can introduced a convolution product, which turns out to be commutative, and the Coulomb branch is the scheme $\mathcal{M}_C(G,M):=Spec \, H_*^{G_\mathcal{O}}$ with this product. The remarkable fact is:

\begin{theorem}
\emph{\cite{BFN}}
The Coulomb branches $\mathcal{M}_C(G,M)$ are normal affine Poisson varieties of dimension $rk \, G$, and its Poisson structure is symplectic in the smooth locus.
\end{theorem}

\section{Poisson Noether's Problem}
\subsection{Proof of Theorem 1.4}
We begin this section with a simple lemma that gives a criterion of recognition of Poisson rationaity.

\begin{lemma}\label{recognition}
\emph{\cite{Dumas}}
 If $G$ is a finite group of Poisson automorphisms of $P_n$ and $K_n^G$ is a purely transcedental extension generated by algebraically independent elements $x_1', \ldots, x_n', y_1', \ldots, y_n'$ such that $\{x_i', x_j' \} = \{y_i', y_j' \} =0, \{x_i', y_j'\} = \delta_{ij}, 1 \leq i,j,\leq n$, then $K_n^G$ is isomorphic as Poisson field to $K_n$, isomorphism given by $\phi: K_n^G \rightarrow K_n$, $x_i',y_i' \mapsto x_i, y_i, i=1,\ldots,n$, respectively.
\end{lemma}

Now we recall the following well known facts:

\begin{theorem}\label{CH}
Let $X$ be an affine smooth algebraic variety and $G$ a finite group that acts freely on it. Then the injection $\mathcal{O}(X)^G \rightarrow \mathcal{O}(X)$ induces, by restriction of differential operators, a filtration preserving isomorphism $\theta: D(X)^G \rightarrow D(X/G)$.
\end{theorem}
\begin{proof}
Theorem 3.7 (1) in \cite{CH}.
\end{proof}

\begin{theorem}\label{quantization}
Let $X$ be an affine smooth algebraic variety, and $G$ any finite group of automorphisms of it. It acts in a natural way by Poisson automorphisms on $\mathcal{O}(T^*X)$. Then, with the filtration by order of differential operators, $D(X)$ and $D(X)^G$ are filtered quantizations of $\mathcal{O}(T^*X)$ and $(\mathcal{O}(T^*X))^G$, respectively. In both cases the Poisson bracket has degree $-1$.
\end{theorem}
\begin{proof}
\cite{Schedler}, Proposition 1.5.3.
\end{proof}

From this follow the following general result, which is of independent interest --- a quasi-classical analogue of Theorem \ref{CH}.

\begin{theorem} \label{new-main}
Let $X$ be an affine smooth algebraic variety, and $G$ a finite group that acts freely on it. Then we have an isomorphism of Poisson varieties $\phi: T^*(X)/G \rightarrow T^*(X/G)$.
\end{theorem}
\begin{proof}
Immediate consequence of Theorems \ref{CH} and \ref{quantization}
\end{proof}

With this preparation, we can begin our proof of Theorem \ref{main-1}. The idea of the proof is similar to the one of Noncommutative Noether's Problem for pseudo-reflection groups (\cite{FS}, Section 5.2), in the 'quasi-classical limit'. 

Let $G$ be an irreducible complex reflection group or $S_n$. Consider the invariants $\C[y_1,\ldots,y_n]^G \simeq \C[e_1,\ldots, e_n]$ (by Chevalley-Shephard-Todd Theorem). We introduce the $n \times n$ matrix $M$ with entries $ij$ being $\partial_{y_i} e_j, 1 \leq i,j \leq n$ and consider its determinant $J(M)$. Define $\sigma(M)=J(M)^{|G|}$ in case $G$ is an irreducible complex reflection group, and $\sigma(M)=J(M)^2$ in case it is the symmetric group. $\sigma(M)$ is non-null and $G$-invariant polynomial on the $y_1, \ldots, y_n$ (\cite{Kane}, 20-2, Props. A, B and 21-1, Props. A, B). Write also, in any case, $\sigma(M)'$ as the expression of $\sigma(M)$ as a polynomial in $\C[e_1, \ldots, e_n]$.

We are interested now in the solutions of the systems \[ (*) ML_i = Y_i, \]
where $Y_i$, $i=1, \ldots, n$, is the column vector with $1$ in the $i$th position and $0$ in all others. Let, for each $i$,  \[ L_i =  \left( \begin{array}{c} l_{i1} \\ \vdots \\ l_{in} \end{array}\right) \] be a solution of the linear system  (*). By the Kramer's rule, $l_{ij} \in \C[y_1,\ldots,y_n]_{J(M)}=\C[y_1,\ldots,y_n]_{\sigma(M)}$, $1 \leq i,j \leq n$.

In any case, consider the algebra $P_{n\sigma(M)}$ with its natural Poisson strucutre, as well its invariant subalgebra $P_{n\sigma(M)}^G$. We are going to interpret them as regular functions on a certain contangent bundle and its invariants. Introduce $h= Spec \, \C[y_1,\ldots,y_n]$; $h_r = Spec \, \C[y_1, \ldots, y_n]_{\sigma(M)}$.  $P_{n \sigma(M)}$ is $\C[h_r \times h^*] \simeq \mathcal{O}(T^* h_r)$. $G$ is a finite group of automorphisms that acts freely on $h_r$, since localizing by $\sigma(M)$ remove the reflecting hyperplanes (\cite{Kane}, 20-2, Props. A, B and 21-1, Props. A, B). The induced action of $G$ on $\mathcal{O}(T^*(h_r))$ is exactly the $G$ action on $P_{n\sigma(M)}$. By Theorem \ref{new-main}, we have an isomorphism of Poisson varieties  \[ (\dagger) \, \theta: T^*(h_r)/G \rightarrow T^*(h_r/G)=T^*(h_{r'}), \] $h_r/G = h_{r'}$, where $h_{r}' = Spec \, \C[e_1, \ldots, e_n]_{\sigma(M)'}$. In the level of rings of regular functions, we have by $\theta$ that $\mathcal{O}(T^*(h_r))^G \simeq P_{n\sigma(M)}^G$ is isomorphic as a Poisson algebra to a localization of $P_n$ --- namely $P_{n\sigma(M)'} \simeq \mathcal{O}(T^*(h_{r'}))$. So, in view of lemma \ref{recognition}, Theorem \ref{main-1} (1) holds.

We can find an explicit isomorphism of Poisson algebras using the results in \cite{FS0} and \cite{FS}, where we obtained an explicit isomorphism $D(h_r)^G \simeq D(h_{r'})$, and the fact that this isomophism is a quantization of $\mathcal{O}(T^*(h_r))^G \simeq \mathcal{O}(T^*(h_{r'})) $ - cf. Theorem \ref{quantization}. Namely, using the standard generators $x_1,\ldots,x_n, \partial_1,\ldots,\partial_n$ for the $n$-th Weyl algebra, and noticing that $D(h_{r'})$ is just its localization by $\Delta'$, we have

\begin{proposition}

When $G$ is either the symmetric group or an irreducible complex reflection group, we have that the following map defines an isomorphism between $D(h_r)^G$ and $D(h_{r'})$:

\[ e_i \mapsto x_i;  l_{i1} \partial_1 +  \ldots + l_{in} \partial_n \mapsto \partial_i, i=1,\ldots,n. \]
\end{proposition}
\begin{proof}
By \cite{FS0} Lemma 1 and \cite{FS} Proposition 5.7.
\end{proof}

Hence, if we set for $i=1, \ldots, n$ $y_i' =e_i, x_i' = l_{i1} x_1 + \ldots + l_{in} x_n$, then we have an explicit description of the isomorphism $\theta$ above ($\dagger$):

\begin{proposition} \label{part2}
\[ \theta: P_{n\sigma(M)}^G \rightarrow P_{n\sigma(M)'}, \]
\[ y_i' \mapsto y_i, x_i' \mapsto x_i, i=1, \ldots, n.\]

The extension of this map to the field of fractions gives us an explict isomorphism $\theta: K_n^G \rightarrow K_n$ (cf. lemma \ref{recognition}), and so Theorem \ref{main-1} (2) is proved.
\end{proposition}

\begin{remark}
Localization is necessary. It is allways false that $P_n^G$ is isomorphic to $P_n$ as a Poisson algebra, for \emph{any} finite group acting linearly diagonally (\cite{AF}, Thm. 4).
\end{remark}

\begin{remark}
If we were only interested in showing that $P_n^G$ is rational (instead of Poisson rational), the proof would follow immediatly from the Chevalley-Shephard-Todd Theorem and \cite{Miyata}, remark 3.
\end{remark}

\subsection{Applications of Poisson Noether's Problem}

We recall the following well known result (cf. \cite{FS}, Prop. 5.10):

\begin{proposition} \label{coxeter-like}
Let $W$ be a finite complex reflection group acting on a vector space $V$. Then $W = W_1 \times \ldots \times  W_m$, $V=V_1 \oplus \ldots \oplus V_m \oplus V'$, where each $W_i$ acts as an irreducible complex reflection group on $V_i$, and trivially on $V_j$, $j \neq i$; and $V'$ is fixed by the whole $W$.
\end{proposition}

Let us now prove Corollary \ref{corol-1}.

Let $W$ be an arbitrary finite complex reflection group acting diagonally linear on some $P_n$. By the Proposition above $P_n^W \simeq P_{n_1}^{W_1} \otimes \ldots \otimes P_{n_m}^{W_m} \otimes P_k$, $n_1 + \ldots + n_m + k =n$ ($P_k$ corresponds to the part fixed by the whole $W$). An application of Theorem \ref{main-1} gives our result.

The proof of Corollay \ref{corol-2} is similar. Let $\Gamma < SL_2(\C)$. Let $u,v \in K_2^\Gamma$ be two elements such that the map $K_1^\Gamma \rightarrow K_1$ , $u \mapsto x, v \mapsto y$, is an isomorphism of Poisson fields (cf. \cite{Baudry}). In the wreath product case $G=\Gamma \wr S_n$ action in $P_{n}$, consider the elements $u_i,v_i$, $i=1,\ldots, n$ --- a copy of each $u,v$ in the $n$-factos of the action. Then, repeating the procedure (with $u_i,v_i$ instead of $x_i,y_i$, respectively) in the proof of Theorem \ref{main-1} (2), cf. Proposition \ref{part2}, we obtain our result.

\begin{example}
Here we have an example showing how combining Baudry's results and Thereom \ref{main-1} (2) we can explicity exhibit the isomorphisms. Let $\mathbb{D}_n$ be the binary dihedral group of order $4n$ acting on $\C(x_1,y_1)$, $\{x_1,y_1 \}=1$. Then $\C(x_1,y_1)^{\mathbb{D}_n} \simeq \C(u_1,v_1)$, with:

\[ u_1 = 1/8n((x_1^{-1}y_1)^{-n}-(x_1^{-1}y_1)^{n}) (\frac{ (x_1^{-1}y_1)^{n} -1}{(x_1^{-1}y_1)^{n}+1})^2x_1y_1; v_1= (\frac{ (x_1^{-1}y_1)^{n} +1}{(x_1^{-1}y_1)^{n}-1})^2, \]

and $\{u_1,v_1 \} = 1$ (\cite{Baudry}).

Now lets consider the action of the group $G= \mathbb{D}_n \wr S_3$ on $\C(x_1,y_1, x_2, y_2, x_3, y_3)=K_3$. We have an isomorphism $ \psi: K_3^{\mathbb{D}_n} \rightarrow \C(U_1,V_1,U_2,V_2,U_3,V_3)$. $J(M)$ is $(v_1-v_2)(v_2-v_3)(v_3-v_2)$. Now we define $V_1 = v_1 + v_2 + v_3$; $V_2 = v_1 v_2 + v_2 v_3 + v_3 v_1$; $V_3 = v_1 v_2 v_3$; and

\[ U_1 = \frac{v_1^2(v_2-v_3)}{J} u_1 + \frac{v_2^2(v_3-v_1)}{J} u_2 + \frac{v_3^2(v_1-v_2)}{J} u_3; \]
\[ U_2 = \frac{v_1(v_3-v_2)}{J} u_1 + \frac{v_2(v_1 - v_3)}{J} u_2 + \frac{v_3(v_2-v_1)}{J} u_3;\]
\[ U_3 = \frac{(v_2-v_3)}{J} u_1 + \frac{(v_3-v_1)}{J} u_2 + \frac{ (v_1-v_2)}{J}u_3 . \]

We have $\{V_i,V_j \} = \{ U_i, U_j \} = 0, \{ U_i, V_j \} = \delta_{ij}, i,j =1,2,3$, and the explicit isomorphism $ \psi: K_3^{\mathbb{D}_n} \rightarrow \C(U_1,V_1,U_2,V_2,U_3,V_3)$ is $u_i \mapsto U_i, v_i \mapsto V_i$, $i=1,2,3$ (cf. \cite{FS}, Ex. 5.12).

\end{example}
Finally, considering Calogero-Moser spaces, we need the following result from \cite{Etingof}.

\begin{proposition}
$\mathcal{M}_c(W,h)$ is birationally equivalent to $h \oplus h^* / W$ as a Poisson variety.
\end{proposition}
\begin{proof}
\cite{Etingof}, Proposition 17.7*
\end{proof}

Combining this result with Theorem \ref{main-1}, we immediatly obtain Theorem \ref{main-2}.

We finish this section with a Conjecture:

\begin{conjecture}
For every finite symplectic reflection group $\Gamma$, Poisson Noether's Problem for the  Poisson-Weyl algebra holds.
\end{conjecture}

\section{Birational Equivalence and Poisson Birational Equivalence}

In this section we show that birational equivalence of quotient varieties implies Poisson birational equivalence of the quotient for the induced action of the group on the cotangent bundles.

\begin{proposition}\label{birat}
Let $X,Y$ be two affine smooth varieties, birationally equivalent. Then $\mathcal{O}(T^*X)$ and $\mathcal{O}(T^*Y)$ have isomorphic Poisson field of fractions.
\begin{proof}

Since $X$ and $Y$ are birationally equivalent, there exists $U \subset X$, $V \subset Y$ open subsets that are isomorphic as varieties. We can suppose both of then affine: $U= Spec \, A, V= Spec \, B$, $\mathcal{O}(X) \subset A$, $\mathcal{O}(Y) \subset B$. Then we have $T^*U \simeq T^*V$ as Poisson varieties. Since we have $\mathcal{O}(T^*X) \subset \mathcal{O}(T^*U); \, \mathcal{O}(T^*Y) \subset \mathcal{O}(T^*V)$  as Poisson subalgebras with the same field of fractions, we are done.
\end{proof}
\end{proposition}

In case of smooth curves, we have a converse:

\begin{proposition}
Let $X, Y$ be two affine smooth curves, sucht that $\mathcal{O}(T^*X)$ and $\mathcal{O}(T^*Y)$ have isomorphic Poisson field of fractions. Then they are birationally equivalent.
\begin{proof}
Let $Z$ be an affine smooth variety of dimension $n$. Since the sheaf of Khaler differentials is locally free, the contangent bundle is locally trivial --- namely, there is a dense affine open set $U \subset Z$ such that $T^*(U) = \mathcal{O}(U)[\xi_1, \ldots, \xi_n]$; and $Frac \, \mathcal{O}(Z) = Frac \, \mathcal{O}(U)$.  So, if $X$ and $Y$ are two curves whose cotangent bundles are birationally equivalent, then $Frac \, \mathcal{O}(X)(\xi) \simeq Frac \, \mathcal{O}(Y)(\xi')$. So the two curves are stably birational equivalent (cf. \cite{JJ}), and hence birationally equivalent (\cite{Hartshorne}, V, Exercise 2.1).
\end{proof}
\end{proposition}

It seems an interesting question to understand whether the converse of Proposition \ref{birat} holds for varieties of higher dimension.

\begin{lemma}\label{lemma-new}
Let $X$ be an affine smooth variety, $0 \neq f \in \mathcal{O}(X), X_f= Spec \, \mathcal{O}(X)_f$. $\mathcal{O}(T^*X)_f \simeq \mathcal{O}(T^*X_f)$.
\end{lemma}
\begin{proof}
Let $\Theta_X$ be the Lie algebra of vector fields on $X$. It is an $\mathcal{O}(X)$-module. It is well known that $\Theta_{X_f}$ is isomorphic to $\mathcal{O}(X)_f \otimes_{\mathcal{O}(X)} \Theta_{X_f}$ as an $\mathcal{O}(X)_f$-module (\cite{MR}, 15.1.24). Since, for an affine smooth variety $Y$, $\mathcal{O}(T^*Y) = Sym_{\mathcal{O}(Y)} \Theta_Y$, we are done.
\end{proof}

\begin{theorem}\label{quasi-classical}
Let $G$ be a finite group acting on an affine smooth variety $X$, such that $X/G$ is birationally equivalent to an affine smooth variety $Y$. Then $\mathcal{O}(T^*X)^G$ and $\mathcal{O}(T^*Y)$ have isomorphic Poisson field of fractions.
\begin{proof}
There exists a principal open subset of $X$, $X_f$, where the action of $G$ is free (cf. lemma 4.4 \cite{FS}). By Lemma \ref{lemma-new}, we have $(\mathcal{O}(T^*X)^G)_f =(\mathcal{O}(T^*X)_f)^G \simeq \mathcal{O}(T^*X_f)^G$. By Theorem \ref{new-main}, this is isomorphic to $\mathcal{O}(T^*X_f/G)$. By Proposition \ref{birat}, the Poisson field of fractions of $\mathcal{O}(T^*X_f/G)$ is the same as the one of $\mathcal{O}(T^*Y)$. But the former has the same Poisson field of fractions as $\mathcal{O}(T^*X)^G$ up to isomorphism.
\end{proof}
\end{theorem}

Theorem \ref{main-1} (1) is particular case of this last result, since by Chevalley-Shephard-Todd Theorem, the quotient of an affine space by a finite complex reflection group is birational to the affine space itself.

\subsection{Coulomb branches and trigonometric Cherednik algebras}

In this subsection, as applications of the above results, we prove Theorems \ref{main-3} and \ref{main-4}.

First we consider Coulomb branches.

\begin{theorem}
Let $H$ be a maximal torus of $G$. $\mathcal{M}_C(G,M)$ is birationally equivalent to $T^*(H^\vee/W)$ as a Poisson variety.
\end{theorem}
\begin{proof}
\cite{BFN}, Corol. 5.21.
\end{proof}

This already implies that the birational equivalence class of the Coulomb branch depends only on $G$. The quotient of $H^\vee$ by $W$, since the Weyl group acts by automorphisms of algebraic groups on $H^\vee$, is again an affine connected algebraic group; a connected linear algebraic group. Connected linear algebraic groups over the complex numbers are all rational, as is well known. Hence $H^\vee/W$ is a rational variety. Hence, by Proposition \ref{birat}, we have:

\begin{itemize}
\item[]
For every Coulomb branch, if $\mathcal{A} = \mathcal{O}(\mathcal{M}_C(G,M))$, then the Poisson field of fractions of $\mathcal{A}$ is isomorphic to $K_n = \mathbb{C}(x_1, \ldots, x_s, y_1, \ldots, y_s)$, where $s = rk \, H^\vee$. In other words, the Coulomb branch is Poisson rational.
\end{itemize}

This proves Theorem \ref{main-3}

Now to the proof of Theorem \ref{main-4}. In fact, we are going to show a slightly more general result, about the Cherednik algebra of $X$ an affine variety and $G$ a finite group acting on it (cf. Definition \ref{genCherednik}).

\begin{theorem}
\emph{\cite{Etingof3}}
$U_{t,c}(X,G)$, when $t \neq 0$, is an Ore domain whose quotient ring of fractions is isomorphic to $D(X)^G$. $\mathcal{M}(X,G)$ is Poisson birationally equivalent to $(T^* X)/G$.
\begin{proof}
\cite{Etingof3}, Proposition 2.15 and Theorem 2.29.
\end{proof}
\end{theorem}

\begin{theorem}
In the notation of the above theorem, if $X/G$ is a rational variety, then ($t \neq 0$) $Frac \, U_{t,c}(X,G)$ is isomorphic to quotient ring of fractions of $A_n(\mathbb{C})$, the Weyl algebra, $n= dim X$. $\mathcal{O}(T^*X)^G$ has Poisson field of fractions isomorphic to $K_n$ --- so $\mathcal{M}(X,G)$  is Poisson rational.
\begin{proof}
The first claim follows from \cite{FS}, Thm 1.2. The second one follows from Theorem \ref{quasi-classical}.
\end{proof}
\end{theorem}

Theorem \ref{main-4} follows immediatly, noting that given a Weyl group $W$ with maximal torus $H$, $H/W$ is a rational variety (cf. proof of Theorem \ref{main-3}).

\end{document}